\documentclass[11pt,letterpaper]{amsart}
\usepackage{amssymb}
\usepackage{tikz}
\usepackage{pgfplots}
\usepackage{multirow}
\usepackage{graphicx}
\usepackage{amsmath}
\usepackage{subfigure}
\usepackage{calc}
\usepackage{float}
\usepackage{wrapfig}
\usepackage{lipsum}
\usepackage{pstricks}
\usepackage{caption}
\usepackage{amsfonts}
\usepackage[shortlabels]{enumitem}

\newtheorem{theorem}{Theorem}[section]
\newtheorem{lemma}{Lemma}[section]

\newtheorem{corollary}{Corollary}[section]
\newtheorem{definition}{Definition}[section]
\newtheorem{example}{Example}[section]

\numberwithin{equation}{section}

\def\R{\mathbb{R}}
\def\C{\mathbb{C}}
\def\M{\mathbb{M}}
\def\N{\mathbb{N}}

\def\G{\mathbb{G}}

\def\0{{\bf{0}}}

\def\vx{{\vec{x}}}
\def\vy{{\vec{y}}}

\def\vu{{\vec{u}}}
\def\vv{{\vec{v}}}

\def\vk{{\vec{k}}}
\def\vj{{\vec{j}}}
\def\vi{{\vec{i}}}

\def\va{{\vec{a}}}
\def\vb{{\vec{b}}}
\def\vc{{\vec{c}}}

\def\ve{{\vec{e}}}

\def\ve{{\vec{e}}}
\def\vv{{\vec{v}}}

\def\vu{{\vec{u}}}
\def\vw{{\vec{w}}}
\def\vx{{\vec{x}}}
\def\vy{{\vec{y}}}

\title{Euclidean Algebras}

\author{Xingde Dai}
\address{Xingde Dai, The University of North Carolina at Charlotte, Charlotte, NC 28223, USA}
\email{xdai@uncc.edu}
\author{Wei Huang}
\address{Wei Huang, Wells Fargo Bank}
\email{wei.huang@wellsfargo.com}

\begin{document}

%\allowdisplaybreaks

\begin{abstract}
We introduce a new example of unital commutative $n$-dimensional group algebra $\R_n$ for $n \geq 2$.
The algebra $\R_n$ and the complex numbers $\C$ are astonishingly alike.
The zero divisor set of the algebra has Lebesgue $\mu_n$-measure zero.
The formula for the Haar measure is established.
Also, the analytic function theory in $\R_n,$ for $n=2k$ that similar to the classical theory in $\C$
is introduced.
This includes the Cauchy-Riemann equations, mean-value theorem and Louisville theorem.
\end{abstract}

\keywords{Euclidean space \and commutative algebras \and field theory}
%\PACS{PACS code1 \and PACS code2 \and more}
\subjclass[2000]{MSC12E99 \and MSC13A99}

\maketitle

\section{Introduction} \label{intro}

In this paper, we introduce a new example of unital commutative $n$-dimensional group algebra $\R_n$ for $n \geq 2$.
When $n=2$, the result is the complex numbers system $\C$.
The algebra $\R_n$ and the complex numbers $\C$ are astonishingly alike.
We name $\R_n$ the \textit{Euclidean algebra} because the definition of new multiplication 
works for \textit{all} Euclidean space over $\R$ in a natural way.
The algebra has zero divisors  of Lebesgue $\mu_n$-measure zero.
The formula for the Haar measure is established.
We introduced the analytic function theory in $\R_n,$ for $n=2k$ that similar to the classical theory in $\C$.
This includes the Cauchy-Riemann equations, mean-value theorem and Louisville theorem.
In Section \ref{ss:hn} We compare $\R_n$ with other commutative hypercomplex number systems in literatures.

\section{Multiplication in $\R^n$} \label{ss:multiplication}

Let $n\in\N, n\geq 2$ and $\R^n$ be the $n$-dimensional Euclidean space over $\R$ and $\{\ve_1,\cdots, \ve_n\}$ be the standard basis of $\R^n$.
Define a mapping $\mathfrak{g}:\{\ve_1,\cdots, \ve_n\}\rightarrow \R^n$ as
\begin{align*}
    \left\{
       \begin{array}{llc}
         \mathfrak{g}\ve_1 & =&-\ve_{n}\\
         \mathfrak{g}\ve_2 & = &\ve_1\\
         \cdots & \cdots & \\
         \mathfrak{g}\ve_{n} &= &\ve_{n-1}\\
       \end{array}
     \right.
\end{align*}
Let $I_n$ denote the identity mapping (matrix) on the real $n$-dimensional Euclidean space $\R^n.$
It is clear that $\mathfrak{g}^n=-I_n$ and $\mathfrak{g}^{2n}=I_n$.
The set $\{\mathfrak{g}^\ell \mid \ell =1,2,\cdots, 2n\}$ is cyclic, hence it is an abelian group of order $2n$.
We write $\mathfrak{g}^0 = \mathfrak{g}^{2n} = I_n$.
For $0<\ell\leq n-1$, the matrix form for the mapping $\mathfrak{g}^\ell$ is
\begin{align*}
    \left[\begin{array}{cc}
    O & I_{n-\ell}  \\
    -I_\ell & O'
    \end{array}\right]
\end{align*}
where $I_{n-\ell}$ is the $(n-\ell)\times (n-\ell)$ identity matrix, $I_\ell$ is the $\ell\times \ell$ identity matrix
and $O$ a zero matrix.

Since $\mathfrak{g}^{n+k}= - \mathfrak{g}^{k}$.
The linear span of the group $\mathfrak{G}_n\equiv\{\mathfrak{g}^\ell \mid \ell =1,2,\cdots, 2n\}$ over $\R$ has dimension $n$.
Let $\M_n$ denote this linear span.
The set $\{\mathfrak{g}^0,\mathfrak{g}^1,\cdots, \mathfrak{g}^{n-1}\}$ is a basis for $\M_n$,
 a linear subspace of the $n\times n$ real matrix space.
It is clear that $\M_n$ is closed under matrix multiplication.
Since the group $\{\mathfrak{g}^\ell \mid \ell =1,2,\cdots, 2n\}$ is abelian, the space $\M_n$ with the matrix multiplication is a commutative algebra of dimension $n$
with multiplicative identity $I_n$. Let $A$ be an element in $\M_n$ with $\det A \neq 0$. The inverse matrix $A^{-1}$ exists.
By the Cayley-Hamilton theorem, $A^{-1}$
is a linear combination of $\{I_n, A,A^2,\cdots,A^{n-1}\}$, hence $A^{-1}$ is also an element of $\M_n$.

Define the mapping $\varsigma:$ % \{\ve_1,\ve_2,\cdots,\ve_{n-1},\ve_n\} \rightarrow \{\mathfrak{g}^0,\mathfrak{g}^1,\cdots, \mathfrak{g}^{n-1}\}$
\begin{align*}
 \varsigma : \ve_\ell \mapsto \mathfrak{g}^{\ell-1}, \ell=1,2,\cdots, n.
\end{align*}
The mapping $\varsigma$ is a bijection between bases of $\R^n$ and $\M_n$.
The mapping $\varsigma$ extends to an isomorphism from $\R^n$ onto $\M_n$ as linear spaces over $\R$.
For an element $\vu =[u_1,u_2,\cdots, u_{n-1},u_n] \in \R^n$, we have
\begin{align*}
    \varsigma(\vu) = \varsigma\big(\sum_{\ell=1}^{n} u_\ell \ve_\ell\big) = \sum_{\ell=1}^{n} u_\ell \mathfrak{g}^{\ell-1}\in \M_n.
\end{align*}

For $\vu,\vv \in \R^n$, define $\vu\circledast\vv$ as
\begin{align*}
    \vu\circledast\vv \equiv \varsigma^{-1}\big(\varsigma (\vu) \varsigma(\vv)\big).
\end{align*}

\begin{theorem}\label{thm:rfield}
The Euclidean space $\R^n, n \geq 2$ equipped with the multiplication $\circledast$ is a commutative algebra with multiplicative identity $\ve_1$.
\end{theorem}

\begin{proof}
Let $\vu,\vv,\vw$ be elements in $\R^n$.
\begin{description}
  \item[1. Associative Property]
  By definition we have $\varsigma(\vu\circledast\vv) = \varsigma (\vu) \varsigma(\vv)$ and
  $\varsigma(\vv\circledast\vw) = \varsigma (\vv) \varsigma(\vw)$. Also the multiplication in $\M_n$ satisfies the
  Associative Law. We have
  \begin{align*}
    (\vu\circledast\vv) \circledast \vw &= \varsigma^{-1}\big(\varsigma (\vu\circledast \vv) \varsigma(\vw)\big)\\
                            &= \varsigma^{-1}\big((\varsigma (\vu)\varsigma(\vv)) \varsigma(\vw)\big)\\
                            &= \varsigma^{-1}\big(\varsigma (\vu)(\varsigma(\vv) \varsigma(\vw))\big)\\
                            &= \varsigma^{-1}\big(\varsigma (\vu) \varsigma(\vv \circledast \vw)\big)\\
                            &=     \vu\circledast(\vv \circledast \vw).
  \end{align*}
i.e.
  \begin{align*}
    (\vu\circledast\vv) \circledast \vw = \vu\circledast(\vv \circledast \vw).
  \end{align*}

  \item[2. Commutative Property]
  Since the multiplication in $\M_n$ is commutative, $\varsigma (\vu) \varsigma(\vv)=\varsigma (\vv) \varsigma(\vu)$, we have
  \begin{align*}
    \vu\circledast\vv = \varsigma^{-1}\big(\varsigma (\vu) \varsigma(\vv)\big) =
    \varsigma^{-1}\big(\varsigma (\vv) \varsigma(\vu)\big) = \vv\circledast\vu .
  \end{align*}
  \item[3. Distributive Property]
  Based on the facts that $\varsigma$ and $\varsigma^{-1}$ are isomorphisms of $\R^n$ and $\M_n$ as linear spaces over $\R$,
  also the fact that $\M_n$ is an algebra satisfying the Distributive Law, we have
    \begin{align*}
        \vw \circledast (\vu+\vv) &= \varsigma^{-1}\big(\varsigma (\vw) \varsigma(\vu+\vv)\big)\\
                            &= \varsigma^{-1}\big(\varsigma (\vw) \big(\varsigma(\vu)+\varsigma(\vv)\big)\big)\\
                            &= \varsigma^{-1}\big(\varsigma (\vw) \varsigma(\vu)+\varsigma (\vw)  \varsigma(\vv)\big)\\
                            &= \varsigma^{-1}\big(\varsigma (\vw) \varsigma(\vu)\big)+\varsigma^{-1}\big(\varsigma (\vw) \varsigma(\vv)\big)\\
                            &= \vw \circledast \vu+\vw \circledast \vv.
    \end{align*}
i.e.
    \begin{align*}
        \vw \circledast (\vu+\vv) = \vw \circledast \vu+\vw \circledast \vv.
    \end{align*}

  \item[4. Multiplicative Identity]
  The multiplicative identity in $\M_n$ is $I_n=\mathfrak{g}^0$ and the corresponding element in $\R^n$ is $\varsigma^{-1}(\mathfrak{g}^0)=\ve_1=[1,0,\cdots,0]$.
  We have
    \begin{align*}
        \vu \circledast \ve_1 = \ve_1 \circledast \vu = \varsigma^{-1} \big( I_n  \varsigma (\vu) \big)=\varsigma^{-1} \big( \varsigma (\vu) \big)=\vu.
    \end{align*}
\end{description}
The Theorem \ref{thm:rfield} is proved.
\end{proof}

We call this algebra the \emph{Real Euclidian Algebra}. It is denoted by $\R_n$.
The multiplication table for $\circledast$ is
\begin{equation}\label{eqn:mt}
\begin{tabular}{|l||l|l|l|l|l|l|l|}
  \hline
  $\circledast$ & $\ve_1$ & $\ve_2$ & $\ve_3$ &... & $\ve_{n-2}$& $\ve_{n-1}$ & $\ve_n$ \\
          \hline
          \hline
  $\ve_1$ & $\ve_1$ &$\ve_2$& $\ve_3$ &... & $\ve_{n-2}$& $\ve_{n-1}$& $\ve_n$ \\
  $\ve_2$ & $\ve_2$ &$\ve_3$ & $\ve_4$ &...& $\ve_{n-1}$& $\ve_n$& -$\ve_1$ \\
  $\ve_3$ & $\ve_3$ &$\ve_4$ & $\ve_5$ &...& $\ve_n$& -$\ve_1$& -$\ve_2$ \\
  ... & ...& ...& ...& ... &  ... &  ...&  ...\\
   $\ve_{n-2}$ & $\ve_{n-2}$  & $\ve_{n-1}$ & $\ve_n$ &...& -$\ve_{n-5}$& -$\ve_{n-4}$& -$\ve_{n-3}$  \\
  $\ve_{n-1}$ & $\ve_{n-1}$  & $\ve_n$ & -$\ve_1$ &...& -$\ve_{n-4}$& -$\ve_{n-3}$& -$\ve_{n-2}$  \\
  $\ve_n$ & $\ve_n$ & -$\ve_1$ &-$\ve_2$ & ...& -$\ve_{n-3}$& -$\ve_{n-2}$ & -$\ve_{n-1}$ \\
  \hline
\end{tabular}
\end{equation}

%\vskip 30pt
For $\vx,\vy\in\R_n$, if there is no confusion in the context,
we write $\vx\vy$ for $\vx\circledast\vy$ and $\vx^m$ for the product of $m$-multipliers $\vx\circledast \vx \circledast \cdots \circledast \vx$.
We reserve the notation $\vx \cdot \vy$ for the dot product. For $\va_m, \cdots \va_1,\va_0 \in \R_n$ the polynomial is well defined:
\begin{align*}
    p(\vx)=\va_{m} \vx^{m} +\va_{m-1} \vx^{m-1} +\cdots + \va_{1} \vx +\va_{0}.
\end{align*}

\section{Remarks}

We have some notes in order.
\begin{enumerate}
    \item The algebra $\R_n$ is the Euclidean space $\R^n$ with the standard basis $\{\ve_k \mid k=1,\cdots n\}$ and the basis follows the multiplication table \eqref{eqn:mt}.
    In particular, $\ve_1$ is the multiplicative identity $1$.
    As a Euclidean space over $\R$,
    for $\vx =[x_1,\cdots, x_n]$ and $\va =[a_1,\cdots, a_n]$ in $\R_n$, the distance is
    \begin{align*}
    |\vx-\va|=\sqrt{\sum_{j=1}^n (x_j-a_j)^2}.
    \end{align*}
    Also, the Lebesgue measure $\mu_n$ is naturally defined in $\R_n$.
    \item Let $\vx\in\R_n$, define the norm of $\|\vx\|$ as
        \begin{align*}
            \|\vx\| \equiv \big| \det \varsigma(\vx)  \big|.
        \end{align*}
        Then by the matrix multiplication properties, for $\vx,\vy\in\R_n$
        \begin{align*}
            \|\vx \circledast\vy\| = \big| \det \varsigma (\vx)\det \varsigma(\vy) \big|= \|\vx\| \|\vy\|
        \end{align*}

    \item For $n\geq 2$, the group $\mathfrak{G}_n$ is
    \begin{align*}
    \mathfrak{G}_n=\{\mathfrak{g},\mathfrak{g}^2,\cdots, \mathfrak{g}^{n}=-1, \mathfrak{g}^{n+1}, \cdots,\mathfrak{g}^{2n}=1\}.
    \end{align*}
    When $n=2$,  the part between $\mathfrak{g}$ and $-1$ as well as the part between $\mathfrak{g}^{n+1}$ and $1$ reduced into empty set. So $\mathfrak{G}_2=\{\mathfrak{g},\mathfrak{g}^2,\mathfrak{g}^3,\mathfrak{g}^4\}=\{\mathfrak{g},-1,-\mathfrak{g},1\}$. This is the group
    $\{i,-1,-i,1\}$ where $i$ is the imaginary unit in $\C$. So $\C$ is a natural member of $\R_n$.
    \item Let $\vu=[u_1,\cdots, u_n]\in \R_n$. Then
\begin{align*}
    \vu&=[u_1,\cdots, u_n]\\
    &=\|\vu\| \left[\frac{u_1}{\|\vu\|},\cdots, \frac{u_n}{\|\vu\|}\right]\\
    &=\|\vu\| ~\big[\cos \theta_1,\cdots, \cos \theta_n \big],
\end{align*}
where $\theta_k, k=1,\cdots, n$ is the angle between $\vu$ and $\ve_k$.

    For $n=2$,  we have $\varsigma(\R_2)=\M_2$.
\begin{align*}
   \M_2= \left\{\left[
        \begin{array}{cc}
          a     & b \\
          \textrm{-}b    & a \\
        \end{array}
      \right]
    \mid a,b\in \R\right\}
\end{align*}
Also, $\det \left[
        \begin{array}{cc}
          a     & b \\
          \textrm{-}b    & a \\
        \end{array}
      \right]
     =a^2+b^2$. So an element $[a,b]\in \R_2$ is a zero divisor iff $a=b=0$. i.e. $\R_2$ has no non-trivial zero divisors.
Hence $\R_2$ is a $2$-dimensional commutative division algebra. By the Frobenius theorem (1878) \cite{frobenius}
the algebra $\R_2$ must be isomorphic to $\C$ over $\R$.

We have
\begin{align*}
    \vu =[u_1,u_2] & =\sqrt{u_1^2+u_2^2} \left[\frac{u_1}{\sqrt{u_1^2+u_2^2}},\frac{u_2}{\sqrt{u_1^2+u_2^2}} \right]\\
    &=\sqrt{u_1^2+u_2^2} ~\left[\cos \theta_1,\cos \theta_2 \right],
\end{align*}
Note that in the case $n=2, \theta_2=\frac{\pi}{2}-\theta_1$. This is the formula $z=|z|(\cos \theta+i\sin \theta)$ in complex numbers.

    \item Let $\vu,\vv\in\R_n$ and $\vv_1,\vv_2,\cdots,\vv_n$ be the $n$ row vectors in the matrix $\varsigma(\vv)^\tau$, the transpose matrix of $\varsigma(\vv)$.
    It is clear we have
    \begin{align*}
        |\vv_j|=|\vv|, \forall j\leq n.
    \end{align*}
Then
by the multiplication table we have
\begin{align*}
    \vu\vv=\vu \circledast \vv & = \big[\vu\cdot\vv_1,\cdots,\vu\cdot\vv_n\big]
     = \vu\varsigma(\vv)=\vv\vu,
\end{align*}
the $j^{th}$ component of $\vu \circledast \vv$ is $\vu\cdot \vv_j$, where $``\cdot"$ is the dot product in $\R^n$.
Hence we have

\begin{lemma}\label{lm:linear}
Let $\vu, \vv \in\R_n$, then
\begin{align}\label{eq:linear}
    \vu\circledast\vv = \vu \vv = \vv\varsigma(\vu).
\end{align}
\end{lemma}

%  \item %The construction and Theorem \ref{thm:rfield} in Section \ref{sec:1} still hold if we replace the real number field $\R$ with a field $\F$.
%Consider the product space $\F^n, n\geq 2$. We have
%\begin{theorem}\label{thm:ffield}
%The product space $\F^n$ equipped with the multiplication $\circledast$ is a commutative algebra over $\F$.
%\end{theorem}
%We denote this algebra as $\F_n$.\\
%  \item Let $\F=\C$ and $n\geq 2$ in Theorem \ref{thm:ffield}, we have
%\begin{theorem}\label{thm:cfield}
%The complex Euclidean space $\C^n$ equipped with the multiplication $\circledast$ is a commutative algebra over $\C$.
%\end{theorem}
%We denote this algebra as $\C_n$.\\
%  \item The algebra $\C_2$ over $\C$ is \emph{not} a division algebra.
%For example, $z_0=[1,i]$ is a non-trivial zero divisor since
%$\det\left[\begin{array}{cc}  1 & i \\  -i & 1 \\ \end{array}  \right]=0$.
\end{enumerate}

\section{Zero Divisors}

An element $\vx\in\R_n$ is a zero divisor if and only if $\varsigma (\vx)$ is a zero divisor in $\M_n$ if and only if
  $\det(\varsigma (\vx))$,  a homogeneous polynomial in $n$ variables $\{x_1,\cdots, x_n\}$  has value zero.
Let $\mathfrak{Z}_{_{\R_n}}$ denote the set of zero divisors in $\R_n$.
For a polynomial  $p$, we will write $\mathfrak{Z}(p)$ for $\{\vx  \mid p(\vx) = 0\}$.
\begin{lemma}\label{lm:zero}
Let $p(\vx):\R^n\rightarrow \R$ be a non-zero polynomial on $\R^n$, then $\mu_n(\mathfrak{Z}(p)) = 0.$
\end{lemma}
Here $\mu_n$ is the Lebesgue measure on $\R^n$.
\begin{proof}
We prove the Lemma \ref{lm:zero} by induction on the dimensions $n$.

For $n=1$. Let $p(\vx)$ be a polynomial of degree $m\geq 1$ on $\R$. It's clear $\mathfrak{Z}(p)$ is a finite set hence has Lebesgue measure zero.

Assume the statement is true for all $n=1, \cdots, k-1$.
Consider $n=k$.

Let $p(\vx)=p(x_1,x_2,\cdots,x_k)$ be a non-zero polynomial of degree $m$ on $\R^k$ and let $t\in\R$. Define
\begin{align*}
    p_{_{|x_1=t}}(\vx)\equiv p(t,x_2,\cdots,x_k).
\end{align*}
This is a function in $k-1$ variables $x_2, \cdots, x_k$.
Let $\mathcal{C}$ denote the collection of all points $t$ in $\R$ with the properties that the function $p_{_{|x_1=t}}(\vx)$  is not a zero function.
Apply the induction assumption to $p_{_{|x_1=t}}$, we have
\begin{align*}
    \mu_{k-1}(\mathfrak{Z}(p_{_{|x_1=t}}))=0, \forall t \in \mathcal{C}.
\end{align*}
Let $t\in\R\backslash \mathcal{C}$.
Then $p_{_{|x_1=t}}(\vx)\equiv 0$. So $(x_1-t)$ is a factor of $p$. The polynomial has at most
  finite factors. Hence the set $\R\backslash \mathcal{C}$ is a finite set. So
\begin{align*}
    \mu_{k-1}(\mathfrak{Z}(p_{_{|x_1=t}}))=0, \textrm{for $t\in\R $ a.e.}
\end{align*}
By Fubini't theorem
\begin{align*}
    \mu_k(\mathfrak{Z}(p))&=\int_{\R} \mu_{k-1}(\mathfrak{Z}(p_{_{|x_1=t}}))dt
    = 0.
\end{align*}
The induction completes. The Lemma \ref{lm:zero} is proved.
\end{proof}

\begin{theorem}\label{thm:zero}
Let $n\geq 3, n\in\N$ and $\mu_n$ be the Lebesgue measure on $\R^n$. Then
\begin{enumerate}
  \item The set of zero divisors in $\R_n$, $\mathfrak{Z}_{_{\R_n}}$, is a closed ideal.
  \item $\mu_n(\mathfrak{Z}_{_{\R_n}}) = 0$.
  \item The open set $\R_n\backslash \mathfrak{Z}_{_{\R_n}}$ is a group.
  \item $\mathfrak{Z}_{_{\R_n}}$ is star-shaped in the sense that if $\vx\in \mathfrak{Z}_{_{\R_n}}$,
   then $\{t\vx \mid t \in \R\} \subset \mathfrak{Z}_{_{\R_n}}$.
\end{enumerate}
Let $\G_n$ denote the group $\R_n\backslash \mathfrak{Z}_{_{\R_n}}$.
\end{theorem}
\begin{proof}
1. It is clear that $\det (\varsigma(\vu))$ is a polynomial so continuous. Therefore the set $\{\vu \mid \det(\varsigma(\vu)) \leq \frac{1}{n}\}$ is closed.
This implies that
\begin{align*}
    \mathfrak{Z}_{_{\R_n}} = \bigcap_{n=1}^\infty \{\vu \mid \det(\varsigma(\vu)) \leq \frac{1}{n}\}
\end{align*}
is closed.

2. Let $\vx\in\mathfrak{Z}_{_{\R_n}}$. Then $\vx=[x_1,\cdots,x_n]$ is a zero for the polynomial $\det(\varsigma(\vx))$.
By Lemma \ref{lm:zero}, the conclusion follows.

3. The set $\mathfrak{Z}_{_{\R_n}}$ is closed. So its complement $\G_n$ is open. Let $\va,\vb\in\G_n$.
Then $\varsigma (\va),\varsigma (\vb) \in \M_n$, $\det \varsigma (\va) \neq 0$ and $\det \varsigma (\vb) \neq 0$.
This implies $\det (\varsigma (\va)  \varsigma (\vb)) \neq 0$, or $\va\circledast\vb\notin\mathfrak{Z}_{_{\R_n}}$.
Hence $\G_n$ is closed under multiplication $\circledast$. Notice that $\varsigma(\va)^{-1} \in \M_n$.
Let $\vc$ denote the element $\varsigma^{-1}(\varsigma(\va)^{-1})$ in $\R_n\backslash\mathfrak{Z}_{_{\R_n}}=\G_n $. We have
\begin{align*}
    \va \circledast \vc = \vc \circledast \va =\varsigma^{-1}(\varsigma(\va)^{-1}  \varsigma(\va))=\varsigma^{-1}(I_n)=\ve_n.
\end{align*}
So $\vc$ is the inverse to $\va$. $\G_n$ is a group. Theorem \ref{thm:zero} is proved.
\end{proof}

%\begin{example}\label{ex:41}
%Zero divisors $\mathfrak{Z}_{_{\R_3}}$ in $\R_3$.
%Let $\vu =[u_1,u_2,u_3] \in \R^3$, then $\vu \in \mathfrak{Z}_{_{\R_3}}$
%if and only if $\det(\varsigma(\vu)) = u_1^3 - u_2^3 + u_3^3 +3u_1u_2u_3 = \frac{1}{2}(u_1-u_2+u_3)((u_1+u_2)^2 + (u_2+u_3)^2 +(u_3-u_1)^2) = 0$.
%Hence $\vu \in \mathfrak{Z}_{_{\R_3}}$ if and only if
%\begin{align}\label{eq:zero_divisor_3}
%    \begin{array}{ll}
%       (I) \begin{array}{l}
%u_1-u_2+u_3 = 0
%          \end{array}
%         & \text{ or }
%       (II) \left\{
%          \begin{array}{l}
%          u_1 = t \\
%          u_2 = -t  \\
%         u_3 = t\\
%          \end{array}
%        \right.
%    \end{array}
%\end{align}
%where $t$ is a free real variable.
%The graph of $\mathfrak{Z}_{_{\R_3}}$ is shown in Figure (\ref{fg:r3_zero_divisor}).

%\begin{figure}
%\centering
%\captionsetup{justification=centering}
%\includegraphics[scale=0.4]{r3_zero_divisors.png}
%\caption{Zero Divisors in $\R_3$}\label{fg:r3_zero_divisor}
%\end{figure}

%\end{example}

\begin{example}\label{ex:42}
Zero divisors $\mathfrak{Z}_{_{\R_4}}$ in $\R_4$. Let $\vu =[u_1,u_2,u_3, u_4] \in \R^4$, then $\vu \in \mathfrak{Z}_{_{\R_4}}$
if and only if $\det(\varsigma(\vu)) = (u_1^2-u_3^2 + 2u_2u_4)^2 +(u_4^2 - u_2^2 + 2u_1u_3)^2  = 0 $ if and only if
\begin{align*}
    \begin{array}{ll}
       (I) \left\{
          \begin{array}{l}
            u_1 = s \\
            u_2 = t \\
            u_3 = -s + \sqrt{2}t  \\
            u_4 = -\sqrt{2}s + t  \\
          \end{array}
        \right. & \text{ or }
       (II) \left\{
          \begin{array}{l}
            u_1 = s \\
            u_2 = t \\
            u_3 = s - \sqrt{2}t \\
            u_4 = \sqrt{2}s - t \\
          \end{array}
        \right.
    \end{array}
\end{align*}
\end{example}
where $s$ and $t$ are free real variables.

\section{The Haar measure on $\G_n$}\label{ss:haar}
Let $E$ be a Borel subset of $\G_n$ and $\mu_n$ be the Lebesgue measure on $\R^n$. Define $\nu_n (E)$ as
\begin{align*}
\nu_n(E) \equiv \int_E \frac{1}{\left|\det \varsigma(\vx)\right|} d \mu_n(\vx).
\end{align*}

\begin{theorem}\label{thm:haar}
$\nu_n$ is the Haar measure on $\G_n$ and it is unique up to a scalar multiplication.
\end{theorem}
\begin{proof}
Let $\va\in \G_n$ and $E$ be a Borel subset of  $\G_n$, then
$\va E = E \va \subset \G_n$.
We will show that $\nu_n(\va E)=\nu_n( E)$.
By definition,
\begin{align*}
    \nu_n(\va E) &= \int_{\va E} \frac{1}{\left|\det \varsigma(\vx)\right|} d \mu_n(\vx)\\
    & =\int_{E} \frac{1}{\left|\det \varsigma(\va\vy)\right|} d \mu_n(\va\vy).
\end{align*}
%Consider the mapping $\eta:\R^n\rightarrow \R^n$ by
%$\eta: \vy \mapsto \va\vy, \vy\in\R^n$.  This mapping is linear over $\R$ in the sense that
%$\eta(\alpha \vx+\beta\vy)=\alpha\eta(\vx)+\beta\eta(\vy),\alpha,\beta\in\R$.
By %the definition of multiplication $\circledast$
%%Theorem \ref{thm:rfield} and
%and
Lemma \ref{lm:linear}, we have
\begin{align*}
    \va\vy&=\vy\varsigma(\va).\\
    d\mu_n(\va\vy)&= |\det \varsigma(\va)| d\mu_n(\vy).
\end{align*}
Therefore,
\begin{align*}
    \nu_n(\va E) &=\int_{E} \frac{1}{\left|\det \left(\varsigma(\va) \varsigma(\vy)\right)\right|} d \mu_n(\vy\varsigma(\va))\\
    &=\int_{E} \frac{\left|\det \varsigma(\va)\right|}{\left|\det \varsigma(\va)\right| \left| \det\varsigma(\vy)\right|} d \mu_n(\vy)\\
    &=\int_{E} \frac{1}{\left| \det\varsigma(\vy)\right|} d \mu_n(\vy),
\end{align*}
which is
\begin{align*}
    \nu_n(\va E)=\nu_n(E).
\end{align*}

\end{proof}

\section{Analytic functions in $\R_{2k}$}\label{ss:analytic}

In this section we will introduce the analytic function theory in $\R_{2k}$, which is similar to the analytic function theory in $\C$.

%Let $n\geq 2$ and $\R_n$ be the Euclidian algebra over $\R$.
%This is the Euclidian space $\R^n$ over $\R$ equipped with the multiplication $\circledast$,  as discussed in the previous sections.
%For $\vx =[x_1,\cdots, x_n]$ and $\va =[a_1,\cdots, a_n]$ in $\R_n$, the distance is $|\vx-\va|=\sqrt{\sum_{j=1}^n (x_j-a_j)^2}$.

Let $\Omega$ be an open region in $\R_n$,
and $f$ be a function $f:\Omega\rightarrow \R_n$, that is, for a point $\vx=[x_1,\cdots,x_n]\in\Omega$,
\begin{align*}
    f(\vx)=[f_1(\vx),\cdots,f_n(\vx)]
\end{align*}
where $f_j(\vx), j=1,\cdots,n,$ are \emph{real} valued functions in $n$ \emph{real} variables $x_j, j=1,\cdots, n$.
%Let $\vw=\ve_2$. By the multiplication table, and $\vw^0\equiv\ve_1$, we have
%\begin{align*}
%    f(\vx)= \sum_{j=1} ^n f_j(\vx) \vw^{j-1}.
%\end{align*}

%A function $f$ is  \emph{continuous} on $\Omega$ if
%\begin{align*}
%    \lim _{|\vx- \va|\rightarrow 0} |f (\vx) - f(\va)|=0, \forall \va\in\Omega.
%\end{align*}
%A function $f(\vx)$ is bounded on a neighborhood $b(\va,r)$ of $\va$, an open ball with center $\va$ and radius $r$, if there is a positive number $M$ such that
%\begin{align*}
%    |f(\vx)|\leq M, \forall \vx \in b(\va,r).
%\end{align*}

%\begin{lemma}
%\begin{enumerate}
%  \item $f\vx)$ is bounded on a nationhood of $\va$
%  \item $g(\vx)+h(\vx)$ is continuous at $\va$.
%  \item $g(\vx)h(\vx)$ is continuous at $\va$.
%\end{enumerate}
%\end{lemma}
%
%\begin{lemma}
%$p(\vx)$ is continuous on $\R_n$.
%\end{lemma}
%
%
%Let $\va\in\G_n$. Define $g(\vx)=\va \vx$. Let $\vb\in\R_n$. Then, by Lemma \ref{lm:ineq} we have
%\begin{align*}
%    |g(\vx)-g(\vb)|  &= |\va (\vx-\vb)|\leq \sqrt{n}|\va| |\vx-\vb|
%\end{align*}
%So, the polynomial $g(\vx)=\va\vx$ is continuous on $\R_n$.

Define $\vw=\ve_2 \in \R_n$ and $\vw^0 = \ve_1$. It is clear $\vw \in \G_n$ since $\det(\varsigma(\vw)) \neq 0$.
By the multiplication table, $\vw^{-j} = -\vw^{n-j}$, for $0 \leq j \leq n$.
We have
\begin{align*}
    f(\vx)= \sum_{j=1} ^n f_j(\vx) \vw^{j-1}.
\end{align*}

We say a function $f(\vx)$ is \emph{differentiable} in  $\Omega$ if the limit
\begin{align*}
    \lim_{\vec{\delta}\rightarrow 0, \vec{\delta}\in\G_n} \frac{f(\va+\vec{\delta})-f(\va)}{\vec{\delta}}=f'(\va) \ (\textrm{ or } \frac{df}{d\vx}(\va))
\end{align*}
exists for each point $\va\in\Omega$.

In this definition, the variable  $\vec{\delta}=[x_1,x_2,\cdots,x_n]= \sum_{j=1}^n x_j\vw^{j-1}$ is approaching $0$ while remains in $\G_n$.
Notice that for $x_j \neq 0$, the vectors $x_j \vw^{j-1}$ are in $\G_n$, since $\det \varsigma(\vw^{j-1})\neq 0,  j=1,\cdots, n$.
So in the above definition of
differentiation, one can take the special path along $x_j \vw^{j-1}$ with $x_j \rightarrow 0$:
\begin{align*}
    f'(\va) &= \lim_{\vec{\delta}\rightarrow 0, \vec{\delta}\in\G_n} \frac{f(\va+\vec{\delta})-f(\va)}{\vec{\delta}}
    = \lim_{x_j \rightarrow 0} \frac{f(\va+x_j \vw^{j-1})-f(\va)}{x_j \vw^{j-1}},\\
    f'(\va) &= \vw^{-j+1}\frac{\partial f}{\partial x_j}(\va).
\end{align*}

\begin{definition}
Let $\Omega\subset \R_n$ be an open region and $f$ be a function with domain $\Omega$.
We call $f(\vx)=[f_1(\vx),\cdots,f_n(\vx)]$ an \emph{analytic} function on $\Omega$ if
\begin{enumerate}[(1)]
  \item $f$ is differentiable at each point $\vx\in\Omega$.
  \item The second order partial derivatives of $f_m(\vx)$, $\frac{\partial^2 f_m}{\partial x_i \partial x_j}$ are continuous in $\Omega$.
\end{enumerate}
It is clear $\frac{\partial^2 f_m}{\partial x_i \partial x_j}=\frac{\partial^2 f_m}{\partial x_j \partial x_i}$ by the Clairaut's Theorem.
\end{definition}

To introduce the next results, we set $n=2k$. We have:
\begin{align}\label{eq:deri}
& f'(\va) =  \vw^{-j+1}\frac{\partial f}{\partial x_j}(\va), \text{$~j = 1, ..., 2k,$}
\end{align}
where
\begin{align*}
\frac{\partial f}{\partial x_{j}}(\va)
= \left[\frac{\partial f_1}{\partial x_j},\cdots,\frac{\partial f_{2k}}{\partial x_j} \right].
\end{align*}

For any given $1\leq i \leq k$, we have
\begin{align*}
&f'(\va) =  \vw^{-i+1}\frac{\partial f}{\partial x_i}(\va) = \vw^{-i-k+1}\frac{\partial f}{\partial x_{i+k}}(\va),\\
&-\vw^{k}\frac{\partial f}{\partial x_i}(\va) = \frac{\partial f}{\partial x_{i+k}}(\va),\\
&\left[\begin{array}{cc}
    O & -I_{k}  \\
    I_{k} & O
    \end{array}\right] \frac{\partial f}{\partial x_i}(\va)  = \frac{\partial f}{\partial x_{i+k}}(\va).
\end{align*}

Hence, if $f$ is analytic on $\Omega$, we have
\begin{align}\label{eq:laplacian}
\left\{
  \begin{array}{c}
    \frac{\partial f_{m}}{\partial x_i} = \frac{\partial f_{m+k}}{\partial x_{i+k}}  \\
    \frac{\partial f_{m+k}}{\partial x_i} = -\frac{\partial f_{m}}{\partial x_{i+k}}~
  \end{array}\right.
\forall m,i, 1\leq m,i \leq k.
\end{align}

So we have the following theorem:
\begin{theorem}\label{thm:harmonic}
Given a positive even integer $n=2k$ and an open set $\Omega\subset \R_{2k}$. For any analytic function $f=[f_1,...,f_{2k}]$ on $\Omega$, we have
\begin{enumerate}[(1)]
  \item
$f$ satisfies the Cauchy-Riemann equations
\begin{align*}
\left\{
  \begin{array}{cc}
  \frac{\partial f_m}{\partial x_i} = \frac{\partial f_{m+k}}{\partial x_{i+k}}\\
  \frac{\partial f_{m+k}}{\partial x_i} = -\frac{\partial f_{m}}{\partial x_{i+k}}~
  \end{array}
\right. \text{$\forall~ 1 \leq m,i \leq k$.}
\end{align*}
    \item  For each $f_m, f_{m+k}$, $1\leq m , i\leq k$,
    \begin{align}\label{eq:lap2}
         \frac{\partial^2 f_m}{\partial x_i^2}+ \frac{\partial^2 f_{m}}{\partial x_{i+k}^2}= 0,
    \end{align}
\hskip 10pt and
    \begin{align}\label{eq:lap2b}
         \frac{\partial^2 f_{m+k}}{\partial x_i^2}+ \frac{\partial^2 f_{{m+k}}}{\partial x_{i+k}^2}= 0.
    \end{align}
  \item
  Each $f_m$ is a harmonic function in the sense that $f_m$ has continuous second order partial derivatives and
\begin{align*}
    \sum_{i=1}^{2k}\frac{\partial^2 f_m}{\partial x_i^2} = 0.
\end{align*}
\end{enumerate}
\end{theorem}

\begin{proof}
We will prove the Equation \eqref{eq:lap2} only. The proof to the Equation \eqref{eq:lap2b} is similar to this. Let $m \leq k$. By the Cauchy-Riemann Equations \eqref{eq:laplacian} we have
\begin{align*}
\left\{
  \begin{array}{cc}
  \frac{\partial f_m}{\partial x_i} = \frac{\partial f_{m+k}}{\partial x_{i+k}}\\
  ~\frac{\partial f_{m}}{\partial x_{i+k}}=-\frac{\partial f_{m+k}}{\partial x_i}
  \end{array}
\right. \text{$\forall~ 1 \leq i \leq k$.}
\end{align*}
Hence
\begin{align*}
\left\{
  \begin{array}{cc}
  \frac{\partial^2 f_m}{\partial x_i^2} = \frac{\partial^2 f_{m+k}}{\partial x_{i+k}\partial x_{i}}\\
  ~\frac{\partial^2 f_{m}}{\partial x_{i+k}^2}=-\frac{\partial^2 f_{m+k}}{\partial x_i\partial x_{i+k}}
  \end{array}
\right. \text{$\forall~ 1 \leq i \leq k$.}
\end{align*}
Therefore
\begin{align*}
    \frac{\partial^2 f_m}{\partial x_i^2}+\frac{\partial^2 f_{m}}{\partial x_{i+k}^2}=
    \frac{\partial^2 f_{m+k}}{\partial x_{i+k}\partial x_{i}} -  \frac{\partial^2 f_{m+k}}{\partial x_i\partial x_{i+k}}=
    \frac{\partial^2 f_{m+k}}{\partial x_{i+k}\partial x_{i}} -  \frac{\partial^2 f_{m+k}}{\partial x_{i+k}\partial x_i}=0.
\end{align*}
\end{proof}

\begin{corollary}
Consider the polynomial
function in $\R_{2k}$
\begin{align*}
    p(\vx)&\equiv  \va_m\vx^m +\va_{m-1}\vx^{m-1} \cdots + \va_1 \vx + \va_0\\
    &=[p_1(\vx), \cdots, p_{2k}(\vx)],
\end{align*}
where $\va_m, \va_{m-1}, \cdots, \va_1,\va_0 \in\R_{2k}$.
Then each $p_j(\vx), j=1,\cdots,2k$ is a harmonic polynomial in the sense that they are polynomials and harmonic functions.
\end{corollary}

The following Example \ref{ex:odd} shows that, the requirement of an even dimension in Theorem \ref{thm:harmonic} is necessary.
\begin{example}\label{ex:odd}
Consider $\R_{2k-1}, k > 1$ and the function $f$ given by
\begin{align*}
    f(\vx) \equiv \vx^2.
\end{align*}
The domain of $f$ is $\R_{2k-1}$.
%For $\vx=[x,y,z]=x+y\vw+z\vw^2$ where $\vw=[0,1,0]$.
It is left to the readers to verify that $f$ is analytic and
\begin{align*}
    f'(\vx)=2\vx.
\end{align*}

We have
\begin{align*}
    f(\vx)&=[f_1,\cdots,f_{2k-1}] = \sum_{m=1}^{2k-1} f_m \vw^{m-1} = \left(\sum_{j=1}^{2k-1} x_j\vw^{j-1}\right)^2\\
    & = \sum_{i=1}^{2k-1} x_i\vw^{i-1} \sum_{j=1}^{2k-1} x_j\vw^{j-1}
    = \sum_{i,j=1}^{2k-1}  x_ix_j\vw^{i+j-2}\\
    f_m &=\sum_{\tiny \begin{array}{c} i+j-1=m\\ 1\leq i,j \leq 2k-1\end{array}}  x_ix_j
         -\sum_{\tiny \begin{array}{c} i+j-1=m+2k-1\\1\leq i,j \leq 2k-1\end{array}}  x_ix_j.
\end{align*}

Notice that
\begin{align*}
\left(\sum_{1\leq \ell \leq 2k-1} \frac{\partial^2 }{\partial x_{\ell}^2}\right) x_ix_j
= \left\{
  \begin{array}{cc}
  2 & i = j;\\
  0 & i \neq j.
  \end{array}
\right.
\end{align*}

Thus
\begin{align*}
\lefteqn{\sum_{1\leq \ell \leq 2k-1} \frac{\partial^2 f_m }{\partial x_{\ell}^2}} \\
 = &  \sum_{1\leq \ell \leq 2k-1} \frac{\partial^2 }{\partial x_{\ell}^2}
\left( \sum_{\tiny \begin{array}{c} i+j-1=m\\ 1\leq i,j \leq 2k-1\end{array}}  x_ix_j
      -\sum_{\tiny \begin{array}{c} i+j-1=m+2k-1\\1\leq i,j \leq 2k-1\end{array}}  x_ix_j \right) \\
 = &   \sum_{1\leq \ell \leq 2k-1} \frac{\partial^2 }{\partial x_{\ell}^2}
\left( \sum_{\tiny \begin{array}{c} 2j-1=m\\ 1\leq j \leq 2k-1\end{array}}  x_j^2
      -\sum_{\tiny \begin{array}{c} 2j-1=m+2k-1\\1\leq j \leq 2k-1\end{array}}  x_j^2 \right)    \\
& \text{(Note that the set $\{j \mid 2j-1=m, 1\leq j \leq 2k-1\}$ is empty if $m$ is even.} \\
& \text{ and $\{j \mid 2j-1=m+2k-1, 1\leq j \leq 2k-1\}$ is empty if $m$ is odd. )}  \\
= & \left\{
  \begin{array}{cc}
  m+1 & \text{~m is odd;}\\
  -m-2k & \text{~m is even.}
  \end{array}
\right.
\end{align*}

%The above product consists of $(2k-1)^2$ terms of the form $x_ix_j\vw^{i+j-2}$.
%We denote $g_{i j} \equiv x_i x_j$.
%When $i+j-2=m-1$ the term $g_{i,j}$ is a part of $f_m, \ m=1,\cdots, 2k-1$.
%Consider the following scenarios:
%\begin{enumerate}[(1)]
%  \item If $i\neq j$, then $\frac{\partial^2 g_{ij}}{\partial x_{\ell}^2} = 0,\forall \ell$. So
%  \begin{align*}
%    \sum_{\ell=1}^{2k-1}\frac{\partial^2 g_{ij}}{\partial x_{\ell}^2}=0, \ 1\leq i,j \leq 2k-1; i\neq j.
%  \end{align*}
%  \item If $i = j$, then
%\begin{align*}
%\left\{
%  \begin{array}{cc}
%    \frac{\partial^2 g_{ii}}{\partial x_{\ell}^2} = 0, \text{ when $\ell \neq i$}\\
%    \frac{\partial^2 g_{ii}}{\partial x_{\ell}^2} = 2, \text{ when $\ell = i$.}\\
%  \end{array}
%\right.
%%\text{$\forall~ 1 \leq m,i \leq k$.}
%\end{align*}
%So
%  \begin{align*}
%    \sum_{\ell=1}^{2k-1}\frac{\partial^2 g_{ii}}{\partial x_{\ell}^2}=2, \ i=1, \cdots, 2k-1.
%  \end{align*}
%  \item For $i=1,2,\cdots, k$.
% The term $x_i^2 \vw^{2i-2}$ is in $f_m$ when $m =2i-1=1,3,\cdots, 2k-1$.
%  \item Let $i=k+1,k+2,\cdots, 2k-1$.
%  We have
%  \begin{align*}
%   x_i^2 \vw^{2i-2}= x_i ^2 \vw^{2(i-k)-1} \vw^{2k-1}=-x_i ^2 \vw^{2(i-k)-1}, \ i=k+1,\cdots, 2k-1.
%  \end{align*}
% The term, which is $-x_i ^2 \vw^{2(i-k)-1}$ is in $f_m$ when $m =2,4,\cdots, 2(k-1)$.
%%    We have $x_i^2 \vw^{2i-2}=$
%\end{enumerate}
Hence we have
\begin{align*}
 & \sum_{\ell=1}^{2k-1}\frac{\partial^2 f_m}{\partial x_{\ell}^2} = \left\{
  \begin{array}{cc}
  m+1 & \text{~m is odd}\\
  -m-2k & \text{~m is even}
  \end{array}
\right.   \neq 0.
\end{align*}
So $f_m$ is \emph{not} a harmonic function.
\end{example}

Next we introduce an analytic function theory in $\R_{2k}$, similar to the one in the complex number field $\C$.
We show that in two examples (Theorem \ref{thm:meanvalue} and Theorem \ref{thm:louisville}) that such a generalization is valid.

Let $f=[f_1,\cdots,f_{2k}]$ be an analytic function defined on an open region $\Omega \subset \R_{2k}$.
Let $\va\in \Omega$ and $r>0$ such that the ball $b(\va,r) \subset \Omega$. $\partial b(\va,r)$ is the boundary of the ball.
Let $\nu$ be the probability measure for the sphere.
By Theorem \ref{thm:harmonic} each $f_m, m=1,\cdots, 2k$ is a harmonic function.
It is well known that \cite{axler} for a harmonic function $u$ on $\R^n$ we have
\begin{align*}
    u(\va) = \int_{\partial b(\va,r)} u(\vx) d \nu.
\end{align*}
Apply this mean value theorem to each $f_m$ we have
\begin{align*}
   f(\va)&=\left[f_1(\va),\cdots,f_{2k}(\va)\right]\\
   &= \left[\int_{\partial b(\va,r)} f_1(\vx) d \nu, \cdots, \int_{\partial b(\va,r)} f_{2k}(\vx) d \nu\right]\\
   &= \int_{\partial b(\va,r)} f(\vx) d \nu.
\end{align*}
So we have Theorem \ref{thm:meanvalue}, which is the Mean Value Theorem for analytic functions in $\R_{2k}$.

\begin{theorem}\label{thm:meanvalue}
Let $\Omega \subset \R_{2k}$ be an open region and the ball $b(\va,r) \subset \Omega$ with center $\va\in \Omega$ and radius $r>0$.
$\partial b(\va,r)$ is the boundary of the ball and $\nu$ is the probability measure for the boundary.
For an analytic function $f=[f_1,\cdots,f_{2k}]$ defined on $\Omega$, we have
\begin{align*}
       f(\va) = \int_{\partial b(\va,r)} f(\vx) d \nu.
\end{align*}
\end{theorem}

We call an analytic function an \emph{entire function} on $\R_{2k}$ when it is analytic on $\Omega=\R_{2k}$.
A function $f$ is \emph{bounded} in $\R_{2k}$ if there exists a positive number $M$ such that
\begin{align*}
    |f(\vx)| \leq M, \forall \vx \in \R_{2k}.
\end{align*}

The next Theorem \ref{thm:louisville} is the \emph{Louisville Theorem} for Entire Functions in $\R_{2k}$.
\begin{theorem}\label{thm:louisville}
A bounded entire function in $\R_{2k}$ is a constant function.
\end{theorem}

\begin{proof}
Let $f=[f_1, \cdots, f_{2k}]$ be an entire function on $\R_{2k}$ with a bound $M$.
\begin{align*}
    |f(\vx)| \leq M, \forall \vx\in\R_{2k}.
\end{align*}
Then  for each $m\leq 2k$
\begin{align*}
    |f_m (\vx)| \leq M, \forall \vx\in\R_{2k}.
\end{align*}
Notice that $f_m$ is harmonics.  By the Louisville Theorem for the harmonics functions each $f_m$ is a constant function.
Therefore $f$ is a constant function.
This proves Theorem \ref{thm:louisville}.
\end{proof}

\section{Hypercomplex numbers}\label{ss:hn}

In this section we will justify $\R_n$ in relation to the hypercomplex number system.
``A hypercomplex number is a number having properties departing from those of the real and complex numbers." (\cite{vanderwaerden}, van der Waerden (1985)).
For commutative hypercomplex number systems, the most common commutative examples are the abelian group algebra $K[G]$ where $K$ is a field and $G$ a finite abelian group, the
Davenport's $4$-dimensional commutative hypercomplex number system\cite{davenport}, and  the Jacobi's algebra $J_{\circledast} ^3$. \cite{jacobi}

The algebras $\R_3$ and $J_{\circledast} ^3$ are itentical.  Shlomo Jacobi discovered his $3$-dimensional algebra $J_{\circledast} ^3$  in 2014. His colleagues complete the unfinished manuscript after he passed away.

If we take $K=\R$ and $G=\mathfrak{G}_n$ as in Section \ref{ss:multiplication}, then the real linear span, span$(\mathfrak{G}_n)$ is the real Euclidean space $\R^n$. The group multiplication $\circledast$ will be extended
as multiplication in the linear span and the identity $\ve_1$ will be the multiplicative identity in the group algebra $\R[\mathfrak{G}_n]$ which is denoted by $\R_n$.
We name $\R_n$ the \textit{Euclidean algebra} because this definition of multiplication works for \textit{all} Euclidean space over $\R$ in a natural way.
In our notation, the complex number system $\C$ is $\R[\mathfrak{G}_2]=\R_2$, where the cyclic group $\mathfrak{G}_2$ is $\{i,i^2=-1,i^3=-i,i^4=1\}$.

Let $A$ be the Davenport's hypercomplex number system\cite{davenport}. It has a linear basis $\{1,\vi,\vj,\vk\}$.
An element $\vu$ in $A$ is in the form
\begin{align*}
    \vu = a+b\vi+c\vj+d\vk.
\end{align*}
The multiplication table is described as
\begin{align*}
    &\vi \vj = \vj \vi =\vk;\\
    &\vj\vk=\vk\vj=-\vi; \\
    &\vk\vi=\vi\vk=-\vj;\\
    &\vi^2=\vj^2=-\vk^2=-1.
\end{align*}

In $A$, one has $\vu^2=\pm 1$ for $\vu=1,\vi,\vj,\vk$.
An element in $\R_4$ has the form
\begin{align*}
    \vv = a_1\ve_1+a_2\ve_2+a_3\ve_3+a_4\ve_4.
\end{align*}
To compare $\R_4$ with $A$, we want to find all solutions that satisfies $\vv^2=\pm 1$, or equivalently:
\begin{align*}
 %   \vv^2=\pm 1 ; \textrm{or equivalently }\\
    (a_1\ve_1+a_2\ve_2+a_3\ve_3+a_4\ve_4)^2=\pm 1.
\end{align*}
which is
\begin{align*}
\left\{
  \begin{array}{l}
a_1^2 - 2a_2a_4 - a_3^2 = \pm1\\
2a_1a_2 - 2a_3a_4= 0  \\
2a_1a_3 + a_2^2 - a_4^2 = 0 \\
2a_1a_4 + 2a_2a_3 = 0
  \end{array}
\right.
\end{align*}
The solutions to the above system of equations are:
\begin{align*}
\begin{array}{ll}
\left\{
  \begin{array}{l}
a_1 =\pm1 \\
a_2 = a_3 = a_4 = 0
  \end{array}
\right.,
&~
%\begin{align*}
\left\{
  \begin{array}{l}
a_3 =\pm1 \\
a_1 = a_2 = a_4 = 0
  \end{array}
\right.,
%\end{align*}
\end{array}
\end{align*}
which means there are only $4$ solutions for $\vv$: $\{ \pm\ve_1, \pm\ve_3\}$.
However, they are linearly dependent, hence can not form a basis for $\R_4$.
So,
\begin{theorem}\label{thm:n=4}
It is impossible to find a new basis $\{\vv_1, \vv_1, \vv_2,\vv_4\}$ in $\R_4$ such that
$\vv_1^2 = \vv_2^2 = \vv_3^2 = \vv_4^2=\pm 1$. In other words, $\R_4$ is not algebraically isomorphic to
the Davenport's algebra $A$.
\end{theorem}

The algebra $\R_4$ is a new example of commutative hypercomplex number systerm.
We believe the following statement is true.
\textit{For each natural number $n\geq 3$, it is not possible to find a basis $\{\vb_j \mid j=1,\cdots, n\}$ for the algebra $\R_n$ with properties that
\begin{align*}
    \vb_j^2=\pm1, j=1,\cdots, n.
\end{align*}}

\section{The Hypercomplex Number System $\R_4$}\label{ss:4d}
In this last section we will list the basic facts of the $4$-dimensional hypercomplex number system $\R_4$.
\begin{enumerate}
    \item $\R_4$ is a real Euclidean space with basis $\{\ve_1,\ve_2,\ve_3, \ve_4\}$.
         An element $\vv=[v_1,v_2,v_3,v_4],v_i\in\R$ in $\R_4$ has the form
        \begin{align*}
            \vv= v_1\ve_1 + v_2\ve_2 + v_3\ve_3 + v_4\ve_4.
        \end{align*}
The distance between $\vu=[u_1,u_2,u_3,u_4]$ and $\vv=[v_1,v_2,v_3,v_4]$ is
        \begin{align*}
            |\vu-\vv| =  \sqrt{(u_1-v_1)^2 + (u_2-v_2)^2 + (u_3-v_3)^2+ (u_4-v_4)^2}.
        \end{align*}
%    \item It is well known the the Euclidean space $\R_3$ equipped with the cress product $\times$ as multiplication is
%the Lee algebra.

    \item The multiplication table \eqref{eqn:mt} for this basis becomes
\begin{equation*}
\begin{tabular}{|l||l|l|l|l|}
  \hline
  $\circledast$ & $\ve_1$ & $\ve_2$ & $\ve_3$ & $\ve_4$ \\
          \hline
          \hline
  $\ve_1$ & $\ve_1$ &$\ve_2$& $\ve_3$ & $\ve_4$ \\
  $\ve_2$ & $\ve_2$ &$\ve_3$ & $\ve_4$ & -$\ve_1$ \\
  $\ve_3$ & $\ve_3$ &$\ve_4$ & -$\ve_1$ & -$\ve_2$ \\
  $\ve_4$ & $\ve_4$ & -$\ve_1$ &-$\ve_2$ & -$\ve_3$ \\
  \hline
\end{tabular}
\end{equation*}
    \item By Lemma \ref{lm:linear}
        \begin{align*}
            \vu\circledast\vv&=\vv\circledast\vu&=\\
            &= \vv \varsigma (\vu) = \vu \varsigma(\vv)\\
            &= [v_1,v_2,v_3,v_4]
            \left[
              \begin{array}{cccc}
                u_1  &  u_2 & u_3 & u_4\\
                \textrm{-}u_4 &  u_1 & u_2 & u_3\\
                \textrm{-}u_3 &  \textrm{-}u_4 & u_1 & u_2\\
                \textrm{-}u_2 & \textrm{-}u_3 &  \textrm{-}u_4 & u_1\\
              \end{array}
            \right]\\
            &= [u_1,u_2,u_3,u_4]
                \left[
              \begin{array}{cccc}
                v_1  &  v_2 & v_3 & v_4\\
                \textrm{-}v_4 &  v_1 & v_2 & v_3\\
                \textrm{-}v_3 &  \textrm{-}v_4 & v_1 & v_2\\
                \textrm{-}v_2 & \textrm{-}v_3 &  \textrm{-}v_4 & v_1\\
              \end{array}
                \right]
        \end{align*}
    \item Define
        \begin{align*}
            \| \vu \| = \left| \det \varsigma(\vu) \right|.
        \end{align*}
Then
        \begin{align*}
            \| \vu\circledast\vv\|=\| \vu\|\ \| \vv\|.
        \end{align*}

    \item By Example \ref{ex:42}, the zero divisors of $\R_4$ are the union of 2 planes:
    \begin{align*}
       \mathfrak{Z}_{_{\R_4}}= & \{[s,t,-s + \sqrt{2}t,-\sqrt{2}s + t] \mid s,t \in \R \} \\
       & \cup \{[s,t,s - \sqrt{2}t,\sqrt{2}s - t] \mid s,t \in \R \}.
    \end{align*}
    \item Let $E$ be a Borel subset of $\G_4=\R_4\backslash \mathfrak{Z}_{_{\R_4}}$.
    By Theorem \ref{thm:haar}, up to a constant multiplier, the Haar measure $\nu_4$ is,
        \begin{align*}
            \nu_4(E) &= \int_E \frac{1}{\left|\det \varsigma(\vu)\right|} d \mu_4(\vu)\\
            &=\int_E \frac{1}{(u_1^2-u_3^2 + 2u_2u_4)^2 +(u_4^2 - u_2^2 + 2u_1u_3)^2} d \mu_4,
        \end{align*}
    where $\mu_4$ is the Lebesgue measure in $\R^4$.
%    \item Using the reasoning as in Theorem \ref{thm:n=4}, one can prove that $\R_3$ does not have a basis $\{\vv_1,\vv_2,\vv_3\}$
%    with the properties that $\vv_j^2=\pm1, j=1,2,3$.
\end{enumerate}

\bibliographystyle{amsplain}

\end{document}